\documentclass[12pt, reqno]{amsart}
\usepackage{floatrow}
\usepackage{amssymb}
\usepackage{amsfonts}
\usepackage{graphicx}   
\usepackage{amsthm}
\usepackage{amsmath}
\usepackage{amscd}
\usepackage[latin2]{inputenc}
\usepackage{t1enc}
\usepackage[mathscr]{eucal}
\usepackage{indentfirst}
\usepackage{pict2e}
\usepackage{epic}
\numberwithin{equation}{section}
\usepackage{stmaryrd}
\usepackage{scalerel,stackengine}
\usepackage[margin=2.5cm]{geometry}
\usepackage{mathtools}
\usepackage{tikz-cd}
\usepackage{tikz}
\usepackage[colorlinks = true,
            linkcolor = blue,
            urlcolor  = blue,
            citecolor = blue,
            anchorcolor = blue]{hyperref}

\usepackage[titletoc]{appendix}
\usepackage[T1]{fontenc}
\usepackage[numbers,square]{natbib}

\usepackage{setspace}
\usetikzlibrary{calc}
\usetikzlibrary{decorations.pathreplacing}

\makeatletter
\DeclareRobustCommand*\cal{\@fontswitch\relax\mathcal}
\makeatother

\theoremstyle{plain}
\newtheorem{Th}{Theorem}[section]
\newtheorem{Lemma}[Th]{Lemma}
\newtheorem{Cor}[Th]{Corollary}
\newtheorem{Prop}[Th]{Proposition}

\theoremstyle{definition}
\newtheorem{Def}[Th]{Definition}

\newtheorem{Rem}[Th]{Remark}
\newtheorem{?}[Th]{Problem}

\usepackage{dsfont}
\usepackage{graphicx}
\usepackage{times}
\usepackage{enumerate}
\usepackage{pgf,tikz}

\usepackage{tikz-cd} 
\usetikzlibrary{math}
\usetikzlibrary{arrows,decorations.markings}

\usepackage{adjustbox}

\stackMath
\newcommand\reallywidehat[1]{%
\savestack{\tmpbox}{\stretchto{%
  \scaleto{%
    \scalerel*[\widthof{\ensuremath{#1}}]{\kern-.6pt\bigwedge\kern-.6pt}%
    {\rule[-\textheight/2]{1ex}{\textheight}}
  }{\textheight}%
}{0.5ex}}%
\stackon[1pt]{#1}{\tmpbox}%
}
\stackMath
\newcommand\reallywidetilde[1]{%
\savestack{\tmpbox}{\stretchto{%
  \scaleto{%
    \scalerel*[\widthof{\ensuremath{#1}}]{\kern-.6pt\sim\kern-.6pt}%
    {\rule[-\textheight/2]{1ex}{\textheight}}
  }{\textheight}%
}{0.5ex}}%
\stackon[1pt]{#1}{\tmpbox}%
}
\def\CC{\mathbb{C}}

\def\RR{\mathbb{R}}
\def\ZZ{\mathbb{Z}}
\def\ds{\displaystyle}

\def\a{\alpha}
\def\b{\beta}
\def\g{\gamma}

\def\w{\omega}
\def\L{\Lambda}
\def\l{\lambda}
\def\ra{\rightarrow}
\def\tt{\theta}
\def\LL{\mathbb{L}}
\def\zt{\zeta}

\def\mb{\mathbf}
\def\ls{{\cal E}}

\def\fr{F}
\def\ac{\Phi}
\def\LG{{\cal L}}

\def\ll{\ell}
\def\Moduli{\overline{{\cal M}}_{k+1}(\overrightarrow{L},\overrightarrow{C},\b,J)}
\def\pr{\text{pr}}
\newcommand{\tw}[1]{#1^{\zt}}
\def\tww{(\cdot)^{\zt}}
\newcommand{\energy}[1]{E(#1)}
\def\diifeo{{\cal F}}

\newcommand{\xdashrightarrow}[2][->]{
\tikz[baseline=-\the\dimexpr\fontdimen22\textfont2\relax]{
\node[anchor=south,font=\scriptsize, inner ysep=1.5pt,outer xsep=2.2pt](x){#2};
\draw[shorten <=3.4pt,shorten >=3.4pt,dashed,#1](x.south west)--(x.south east);
}
}

\makeatletter

\newcommand*\bigcdot{\mathpalette\bigcdot@{1}}
\newcommand*\bigcdot@[2]{\mathbin{\vcenter{\hbox{\scalebox{#2}{$\m@th#1\bullet$}}}}}
\makeatother
\def\O{\Omega^{\bigcdot}}
\def\H{H^{\bigcdot}}
\makeatother

\begin{document}
\title{Cyclic group actions on Fukaya categories and mirror symmetry}

\author{Chi Hong Chow}

\address{The Institute of Mathematical Sciences and Department of Mathematics, The Chinese University of Hong Kong, Shatin, Hong Kong}

\email{chchow@math.cuhk.edu.hk}

\author{Naichung Conan Leung}

\address{The Institute of Mathematical Sciences and Department of Mathematics, The Chinese University of Hong Kong, Shatin, Hong Kong}

\email{leung@math.cuhk.edu.hk}

\begin{abstract}  
Let $(X,\w)$ be a compact symplectic manifold whose first Chern class $c_1(X)$ is divisible by a positive integer $n$. We construct a $\ZZ_{2n}$-action on its Fukaya category $Fuk(X)$ and a $\ZZ_n$-action on the local models of its moduli of Lagrangian branes. We show that this action is compatible with the gluing functions for different local models.
\end{abstract}


\maketitle


\section{Introduction} 
Let $(X,\w)$ be a compact symplectic manifold. We study the following objects:
\begin{enumerate}
\item the Fukaya category $Fuk(X)$ of $X$; and
\item the moduli ${\cal M}$ of Lagrangian branes on $X$ with the superpotential $W$. 
\end{enumerate}

This paper establishes results about cyclic group actions on these objects which arise from the divisibility of the first Chern class $c_1(X)$. 

\noindent \textbf{Assumption (A).} The first Chern class $c_1(X)\in H^2(X;\ZZ)$ is divisible by a positive integer $n$. 

Let us begin with (1). Let $\zt$ be a complex number. Define a \textit{$\zt$-twisted $A_{\infty}$ functor} on $Fuk(X)$, or simply a \textit{twisted $A_{\infty}$ functor}, to be an $A_{\infty}$ functor of the form
\[\ac:Fuk(X)\ra Fuk(X)_{(\zt)}\]
where $Fuk(X)_{(\zt)}$ is the $A_{\infty}$ category whose objects and morphism spaces are the same as those of $Fuk(X)$, and whose $A_{\infty}$ product $(m_{(\zt)})_k$ is defined by 
\[ (m_{(\zt)})_k := \zt^{k-2}m_k,~k\geqslant 0\]
where $m_k$ is the $A_{\infty}$ product of $Fuk(X)$. Clearly, a twisted $A_{\infty}$ functor can also be regarded as an $A_{\infty}$ functor $Fuk(X)_{(\zt^i )}\ra Fuk(X)_{(\zt^{i+1})}$ for any $i\in\ZZ$.

\begin{Th}\label{thm} Assume (A). Put $\zt =e^{\frac{2\pi i}{2n}} $. There exists a $\zt$-twisted $A_{\infty}$ functor $\ac$ on $Fuk(X)$ whose $(2n)$-th power is $A_{\infty}$ homotopic to the identity functor $\text{id}_{Fuk(X)}$.
\end{Th}

Next we consider (2). Recall that ${\cal M}$ comes with the \textit{superpotential} $W$ which arises from the effect by quantum corrections. It is known that local models for $({\cal M},W)$ are the \textit{weak Maurer-Cartan schemes} ${\cal M}_{weak}(L)$ associated to Lagrangians $L$ of $X$ \cite{CO, FOOO, FOOO toric}. The deformed $m_0$ of each $\mb{b}\in {\cal M}_{weak}(L)$ is by definition equal to the unit class of $CF(L,L)$ multiplied by a constant which is the value of $W$ evaluated at $\mb{b}$.

\begin{Prop}\label{localaction} Assume (A), then there is a $\ZZ_n$-action on $({\cal M}_{weak}(L),W)$, i.e.  
\[W(\tau \cdot \mb{b})=e^{\frac{2\pi i}{n}} W(\mb{b})\text{ for any }\mb{b}\in{\cal M}_{weak}(L) \]
where $\tau$ is a generator of the $\ZZ_n$-action on ${\cal M}_{weak}(L)$.
\end{Prop} 

In order to give a reasonable structure on ${\cal M}$, one has to define gluing functions between the weak Maurer-Cartan schemes associated to two different Lagrangians. This problem has been studied by a lot of people \cite{Auroux2, CLL, CHL1, CHL3, Futrick, HKL, KS, Pascal, Seidel1}. In this paper, we consider the following effective approach by Fukaya \cite{Futrick} whose idea is now known as the \textit{Fukaya's trick}. For any two Lagrangians $L$ and $L'$ which can be brought from one to the other by an isotopy $\varphi_t$. Fix an $\w$-tame almost complex structure $J$. Then the count of $(\varphi_t^{-1})_*J$-holomorphic disks bounding $L$ yields the desired gluing function 
\begin{equation}\label{wall-crossing}
\Psi_{L,L'}: {\cal M}_{weak}(L) \dasharrow {\cal M}_{weak}(L')
\end{equation}
where the dash arrow indicates that this function is defined only on an open subset of the domain. See Section \ref{recall} for more detail.

\begin{Prop}\label{commutewithwallcrossing} The $\ZZ_n$-actions on ${\cal M}_{weak}(L)$ and ${\cal M}_{weak}(L')$ commute with $\Psi_{L,L'}$ in \eqref{wall-crossing}.
\end{Prop}

In other words, our $\ZZ_n$-action on each $({\cal M}_{weak}(L),W)$ is compatible with the gluing of these local models. Hence, we obtain

\begin{Th} There exists a $\ZZ_n$-action on $({\cal M},W)$.
\end{Th}

As an example, assume $X$ is K\"ahler and has an anticanonical divisor $D$. Consider an SYZ fibration defined on the complement $X-D$, i.e. a special Lagrangian torus fibration with singularities \cite{SYZ}. By gluing the local models associated to the smooth Lagrangian torus fibers based on the Fukaya's trick, the moduli ${\cal M}$ of these Lagrangian tori can thus be given the structure of an analytic variety\footnote{In general, ${\cal M}$ is a \textit{rigid analytic space}, a notion first brought by Kontsevich and Soibelman \cite{KS} into the picture of mirror symmetry.}. The details can be found in the work of Tu \cite{Tu} under the assumption $W=0$ and the recent work by Yuan \cite{Yuan} for the general case. See also the work of Abouzaid \cite{Ab1,Ab2,Ab3} using another approach. In this case $({\cal M},W)$ is the (uncompactified) SYZ mirror of the pair $(X,D)$ which we denote by $(\check{X}^{\circ},W)$.

\begin{Cor}\label{cor} There exists a $\ZZ_n$-action on $(\check{X}^{\circ},W)$.
\end{Cor}

\begin{Rem}\label{Riem} If $X$ is Fano, then $(\check{X}^{\circ},W)$ is usually defined over $\CC$ and can be compactified to the complete mirror $(\check{X},W)$ which is an affine variety by adjoining a codimension-two subvariety (those points arising from the singular fibers). We point out that in this case our $\ZZ_n$-action on $(\check{X}^{\circ},W)$ can be extended to a $\ZZ_n$-action on the complete mirror $(\check{X},W)$ by the \textit{second Riemann extension theorem}\footnote{It states that if $U$ is an open subset of a normal complex analytic variety $Y$ whose complement $Y-U$ can be locally covered by closed analytic subvarieties of codimension at least two, then every holomorphic function on $U$ extends to a unique holomorphic function on $Y$.}. See Section 4 for more detail.
\end{Rem}

\begin{Rem} As pointed out by Kuznetsov and Smirnov \cite{Ku, Ku2}, the existence of a $\ZZ_n$-action on $(\check{X},W)$ may be interpreted, via the homological mirror symmetry \cite{Kont}, as the mirror of the existence of a $\ZZ_n$-action on the \textit{Lefschetz decomposition} of the derived category $D^b(X)$ of coherent sheaves on $X$. Our results give an A-side interpretation of this phenomenon.
\end{Rem}

This paper is organized as follows. In Section 2, we prove Theorem \ref{thm}. In Section 3, we recall the Fukaya's trick and the definition of ${\cal M}_{weak}(L)$, and prove Proposition \ref{localaction} and \ref{commutewithwallcrossing}. In Section 4, we fill in the details for the claim made in Remark \ref{Riem}.


\section*{Acknowledgements}
We thank Yong-Geun Oh for drawing our attention to \cite{FuGalois, FOOOAn, Oh}. C. H. Chow also thanks Kaoru Ono for useful comments and for teaching him a lot from the monumental book \cite{FOOO}, Cheuk Yu Mak and Weiwei Wu for helpful discussions as well as Junwu Tu and Siu-Cheong Lau for useful email conversations.

This research was supported by grants of the Hong Kong Research Grants Council (Project No. CUHK14301117 \& CUHK14303518) and a direct grant from the Chinese University of Hong Kong (Project No. 4053337).
\section{Action on \texorpdfstring{$Fuk(X)$}{}}\label{actonFuk} 
In this paper, we will not work with a particular version of Fukaya category because, as we will see, our twisted $A_{\infty}$ functor $\ac$ modifies only the local system carried by each object, and hence it does not depend on which approach adopted to handle the issues arising from the moduli spaces of holomorphic disks. We believe that readers can easily apply our ideas to the version they are using. 

Nevertheless, we will recall in Appendix \ref{Fuk(X)} the least amount of features of $Fuk(X)$ which are necessary in order to explain the construction of $\ac$. For example, the objects of $Fuk(X)$ consist of $\LL=(L,\ls)$ where $L$ is an immersed Lagrangian of $X$ with clean self-intersection and $\ls$ is a $\CC^{\times}$-local system on $L$, and the morphism space between two cleanly intersecting objects $\LL_i=(L_i,\ls_i),~i=0,1$ is defined by
\begin{equation}
Hom_{Fuk(X)}(\LL_0,\LL_1):=\bigoplus_{C\in \pi_0(L_0\times_{\iota} L_1)} \O(C;Hom(\ls_0|_C,\ls_1|_C))
\end{equation}
where $\O(C;\ls)$ is any of the standard models (de Rham, singular cochain, etc) of the cohomology group $H^{\bigcdot}(C;\ls)$ with local coefficient $\ls$. 

Here is the idea of the construction of $\ac$. We assign to each Lagrangian $L$ of $X$ a particular $\ZZ_{2n}$-local system $\ls_L$ and to each connected component $C$ of the fiber product $L_0\times_{\iota} L_1$ a flat section $f_C$ of $Hom(\ls_{L_0}|_C,\ls_{L_1}|_C)$. Then $\ls_L$ and $f_C$ will contribute to the object and morphism parts of $\ac$ respectively. The key is to show that $\ac$ satisfies the twisted version of $A_{\infty}$ equations.

\begin{Rem} The idea of twisting objects by particular $\CC^{\times}$-local systems has been used by Fukaya in his early work \cite{FuGalois}. They are the restriction of a prequantum line bundle on the ambient manifold. See also \cite{FOOOAn} and \cite{Oh} for other applications of these local systems. Recently, Auroux and Smith \cite{AS} constructed group actions on the Fukaya categories of Riemann surfaces using ambient local systems.

We point out that our local system $\ls_L$ is different from theirs, as it cannot be extended to an ambient one and it is equal to Fukaya's one only when $X$ and $L$ are \textit{monotone}. 
\end{Rem}

\begin{Rem}
Regarding group actions on Fukaya categories, we also mention the work \cite{CH} of Cho and Hong. But their action is not constructed in the above fashion, i.e. twisting objects by local systems.
\end{Rem}
\subsection{Object level}\label{object level} Let $m=\frac{1}{2}\text{dim}(X)$. Notice that the unitary group $U(m)$ admits a unique $\ZZ_n$-covering group $\reallywidetilde{U(m)}$, i.e. there is a short exact sequence
\[0\ra \ZZ_n \ra \reallywidetilde{U(m)}\ra U(m)\ra 1.\]

Recall we have imposed the condition that $c_1(X)$ is divisible by $n$. This means that the structure group of the frame bundle of $X$ is reduced from $U(m)$ to $\reallywidetilde{U(m)}$. Let $LG(m)$ be the Lagrangian Grassmannian of $(\CC^m,\w_{std})$ which has fundamental group $\ZZ$. We know that
\[LG(m)\simeq U(m)/O(m)\]
is a symmetric space of $U(m)$, or $U(m)/\{\pm 1\}$. As $\reallywidetilde{U(m)}$ is a $\ZZ_{2n}$-cover of $U(m)/\{\pm 1\}$, it acts on the $\ZZ_{2n}$-cover $LG'(m)$ of $LG(m)$. Since the reduced frame bundle has structure group $\reallywidetilde{U(m)}$, we have

\begin{Lemma} \cite[Lemma 2.2]{Seidel2} The Lagrangian Grassmannian bundle $\LG_X:=LG(TX,\w)$ on $X$ admits a fiberwise cover $\LG'_X\ra\LG_X$ with deck transformation group isomorphic to $\ZZ_{2n}$. 
\end{Lemma}

Let $\iota: L \ra X$ be a Lagrangian immersion with clean self-intersection. The \textit{Gauss map} of $L$ is a section $\tt_L$ of $\iota^*\LG_X$ defined by
\begin{equation}\label{Lagsection}
\tt_L(x):=d\iota (T_xL)\in (\LG_X)_{\iota(x)},~x\in L.
\end{equation}
The inverse image of the subspace $\tt_L(L)\subseteq \iota^*\LG_X$ under the fiberwise covering map $\iota^*\LG'_X\ra \iota^*\LG_X$ is then a $\ZZ_{2n}$-local system on $L$, which we denote by $\ls_L$. We may regard $\ls_L$ as a $\CC^{\times}$-local system via the inclusion $\ZZ_{2n}\hookrightarrow \CC^{\times}:1~(\bmod{2n})\mapsto \zt:= e^{\frac{2\pi i}{2n}}$.

\begin{Def} Let $\LL=(L,\ls)$ be an object of $Fuk(X)$. Define
\[\ac(\LL):=(L,\ls\otimes\ls_L).\]
\end{Def}

\begin{Rem}\label{order=n} If $L$ is oriented, then $\tt_L$ has a lift in the fiberwise double cover of $\iota^*\LG_X$ which lies between $\iota^*\LG_X$ and $\iota^*\LG'_X$. It follows that the $\ZZ_{2n}$-local system $\ls_L$ is reduced to a $\ZZ_n$-local system, and hence $\ac^n(\LL)=\LL$.
\end{Rem}
\subsection{Morphism level}\label{mor level} The linear part $\ac_1$ of our twisted $A_{\infty}$ functor $\ac$ is of the form
\[\ac_1:Hom_{Fuk(X)}(\LL_0,\LL_1)\ra Hom_{Fuk(X)}(\ac(\LL_0),\ac(\LL_1)). \]
By \eqref{mor}, the morphism space between two objects $\LL_i=(L_i,\ls_i),~i=0,1$ of $Fuk(X)$ is given by 
\begin{equation}\label{hom1}
 Hom_{Fuk(X)}(\LL_0,\LL_1)=\bigoplus_{C\in \pi_0(L_0\times_{\iota} L_1)}\O(C; Hom(\ls_0|_C,\ls_1|_C)). 
\end{equation}
Then 
\begin{align}
&~ Hom_{Fuk(X)}(\ac(\LL_0),\ac(\LL_1))\nonumber\\
=&\bigoplus_{C\in \pi_0(L_0\times_{\iota} L_1)}\O(C; Hom((\ls_0\otimes \ls_{L_0})|_C,(\ls_1\otimes \ls_{L_1})|_C))\nonumber \\
\simeq &\bigoplus_{C\in \pi_0(L_0\times_{\iota} L_1)}\O(C; Hom(\ls_0|_C,\ls_1|_C)\otimes Hom(\ls_{L_0}|_C,\ls_{L_1}|_C)).\label{hom2}
\end{align}

Define the family version of the \textit{canonical short path} \cite{Auroux}, $\tt^C_t$, taking $\tt_{L_0}|_C$ to $\tt_{L_1}|_C$ through sections of $\iota^*\LG_X$ over $C$ as follows. Consider the symplectic vector bundle $V_C:=TC^{\perp\w}/TC$ defined on $C$ with the induced symplectic form $[\w]$. Its associated Lagrangian Grassmannian bundle $\LG_C:=LG(V_C,[\w])$ embeds canonically into $\LG_X|_C$ through the quotient map $TC^{\perp\w} \twoheadrightarrow V_C$. Notice that the images of $\tt_{L_0}|_C$ and $\tt_{L_1}|_C$ lie in $\LG_C$. Choose a compatible almost complex structure $J_C$ on $(V_C,[\w])$ such that $J_C\cdot TL_0/TC=TL_1/TC$. Then the desired path $\tt^C_t$ is defined by 
\begin{equation}\label{Lagpath}
\tt^C_t:= e^{-\frac{\pi t}{2} J_C}\cdot TL_0/TC,~t\in [0,1].
\end{equation}
\begin{Rem} It can be shown that $\tt^C_t$ is independent of $J_C$ up to homotopy.
\end{Rem}
The lift of $\tt^C_t$ with respect to the fiberwise covering map $\LG'_X|_C\ra \LG_X|_C$ gives an isomorphism $s_C:\ls_{L_0}|_C\ra \ls_{L_1}|_C$ of local systems and hence a flat section of $Hom(\ls_{L_0}|_C,\ls_{L_1}|_C)$ which is denoted by the same notation $s_C$.

By \eqref{hom1} and \eqref{hom2}, $s_C$ induces a chain isomorphism 
\[f_C: \left(\O(C; Hom(\ls_0|_C,\ls_1|_C)) , m_{1,0}\right)\ra \left(\O(C; Hom((\ls_0\otimes \ls_{L_0})|_C,(\ls_1\otimes \ls_{L_1})|_C)),m_{1,0}\right).\]

For later use, we denote by $C'$ the connected component $C$ regarded as an element of $\pi_0(L_1\times_{\iota} L_0)$, i.e. considered by interchanging $L_0$ and $L_1$. Define $s_{C'}$ and $f_{C'}$ similarly. Notice that in this case, we should start with the canonical short path taking $\tt_{L_1}|_C$ to $\tt_{L_0}|_C$. 

Before defining $\ac_1$, recall that our goal is not to define an $A_{\infty}$ functor but a twisted one. That means $\ac_1$ should not commute with $m_{1,0}$ exactly, but commute with it up to a twist. Hence it is natural to introduce the following operator
\begin{Def} Given a chain map $f:A^{\bigcdot}\ra B^{\bigcdot}$. Define $\tw{f}:A^{\bigcdot}\ra B^{\bigcdot}$ by 
\[\tw{f}(a) := \zt^{-r} f(a)\]
for any $a\in A^r,~r\in\ZZ$. Notice that $\tw{f}$ is no longer a chain map.
\end{Def}

\begin{Def} Define 
\[\ac_1 := \bigoplus_{C\in \pi_0(L_0\times_{\iota} L_1)}\tw{f_C}\]
and $\ac_k:=0$ for $k\geqslant 2$.
\end{Def}

\subsection{Proof of Theorem \ref{thm}}
It is clear from the construction that $\ac$ has order $2n$. Thus it remains to verify that $\ac$ satisfies the twisted $A_{\infty}$ equations
\begin{equation}\label{aooequation}
m_k\circ \ac_1^{\otimes k} = \zt^{2-k} \ac_1\circ m_k,~k\geqslant 0.
\end{equation}
Let $\overrightarrow{\LL}=(\LL_0,\ldots,\LL_k)$ be a sequence of $k+1$ objects of $Fuk(X)$. Put $\LL_{-1}:=\LL_k$. Suppose for each $i=0,\ldots,k$, $L_{i-1}$ and $L_i$ intersect cleanly. By \eqref{m_k}, the $A_{\infty}$ product map
\[m_k:\bigotimes_{i=1}^k Hom_{Fuk(X)}(\LL_{i-1},\LL_i)\ra Hom_{Fuk(X)}(\LL_0,\LL_k)\]
is equal to the sum $\ds\sum_{\overrightarrow{C},\b} T^{\energy{\b}} m_{k,\b,\overrightarrow{\LL},\overrightarrow{C}}$ where
\[m_{k,\b,\overrightarrow{\LL},\overrightarrow{C}}: \bigotimes_{i=1}^k \O(C_i; Hom(\ls_{i-1}|_{C_i} ,\ls_i|_{C_i}  ) )\ra \O(C'_0; Hom(\ls_0|_{C'_0} ,\ls_k|_{C'_0}  ) ) \] 
is a multilinear map, $\overrightarrow{C}=(C_0,\ldots,C_k)$ with $C_i\in\pi_0(L_{i-1}\times_{\iota} L_i)$, $\b\in \pi_2\left(X, \bigcup_{i=0}^k\iota(L_i)\right)$ for which $\Moduli\ne\emptyset$ and $E(\b)=\int_{\b}\w$ is the symplectic area.

Put $\ac(\overrightarrow{\LL}):=(\ac(\LL_0),\ldots,\ac(\LL_k))$ and fix $\overrightarrow{C}$. Equation \eqref{aooequation} will be verified if we show that 
\begin{equation}\label{1}
m_{k,\b,\ac(\overrightarrow{\LL}),\overrightarrow{C}}\circ \left(\bigotimes_{i=1}^k \tw{f_{C_i}}\right) =\zt^{2-k} \tw{f_{C_0'}}\circ m_{k,\b,\overrightarrow{\LL},\overrightarrow{C}}.
\end{equation}
Consider inputs $a_i\in\Omega^{r_i}(C_i, Hom(\ls_{i-1}|_{C_i},\ls_i|_{C_i}))$. Then $\tw{f_{C_i}}(a_i)=\zt^{-r_i} f_{C_i}(a_i)$.

\noindent\underline{\textit{Case $(k,\b)=(1,0)$}}

In this case, $\overrightarrow{\LL}=(\LL_0,\LL_1)$, $\overrightarrow{C}=(C_0,C_1)$, and it is necessary that $C_0'=C_1$. The LHS of \eqref{1} is equal to $\zt^{-r_1} m_{1,0,\ac(\overrightarrow{\LL}),\overrightarrow{C}}(f_{C_1}(a_1))$ and the RHS of \eqref{1} is equal to $\zt^{2-1}\cdot \zt^{-(r_1+1)} f_{C_0'}( m_{1,0,\overrightarrow{\LL},\overrightarrow{C}}(a_1))$.
 
Since $-r_1=2-1-(r_1+1)$, $C_0'=C_1$ and $f_{C_1}$ is a chain map, \eqref{1} holds.

\noindent\underline{\textit{Case $(k,\b)\ne(1,0)$}}

The LHS of \eqref{1} is equal to $\zt^{-\sum_{i=1}^k r_i} m_{k,\b,\ac(\overrightarrow{\LL}),\overrightarrow{C}}(f_{C_1}(a_1),\ldots,f_{C_k}(a_k))$ and the RHS of \eqref{1} is equal to $\zt^{2-k-r_0} f_{C_0'}(m_{k,\b,\overrightarrow{\LL},\overrightarrow{C}}(a_1,\ldots, a_k) ) $ where $r_0$ is the degree of $m_{k,\b,\overrightarrow{\LL},\overrightarrow{C}}(a_1,\ldots, a_k)$. By \ref{m_kb}, 
\[r_0=\left(\text{dim}(C_0) -\frac{1}{2}\text{dim}(X)\right) + 2- k-\mu(\b)+ \sum_{i=1}^k r_i.\]

In other words, we have to show
\[ m_{k,\b,\ac(\overrightarrow{\LL}),\overrightarrow{C}}(f_{C_1}(a_1),\ldots,f_{C_k}(a_k)) = \zt^{\mu(\b) -\left(\text{dim}(C_0) -\frac{1}{2}\text{dim}(X)\right) } f_{C_0'}(m_{k,\b,\overrightarrow{\LL},\overrightarrow{C}}(a_1,\ldots, a_k) ) .\]

Observe that the holonomy contribution of a $J$-holomorphic polygon $u\in\Moduli$ to $m_{k,\b,\ac(\overrightarrow{\LL}),\overrightarrow{C}}$ is equal to the tensor product of the holonomy contribution of the same polygon to $m_{k,\b,\overrightarrow{\LL},\overrightarrow{C}}$ and 
\[  PT(\tilde{u}_{[\xi_k,\xi_0]},\ls_{L_k})\circ \left(s_{C_k}\right)_{\tilde{u}(\xi_k)} \circ \cdots \circ \left(s_{C_1}\right)_{\tilde{u}(\xi_1)} \circ PT(\tilde{u}_{[\xi_0,\xi_1]},\ls_{L_0})\in Hom((\ls_{L_0})_{\tilde{u}(\xi_0^+)}, (\ls_{L_k})_{\tilde{u}(\xi_0^-)} )\]
where $PT(\g,\ls)$ is the parallel transport of the local system $\ls$ along the path $\g$ and 
\[\tilde{u}(\xi_i) := (\tilde{u}(\xi_i^-),\tilde{u}(\xi_i^+)) := \left( \lim_{\substack{\xi\ra\xi_i\\\xi\in(\xi_{i-1},\xi_i) }} \tilde{u}_{[\xi_{i-1},\xi_i]}(\xi) , \lim_{\substack{\xi\ra\xi_i\\\xi\in(\xi_i,\xi_{i+1}) }} \tilde{u}_{[\xi_i,\xi_{i+1}]}(\xi)   \right)\in C_i.\] 
The last expression is illustrated schematically as follows.

\begin{center}
\begin{tikzpicture}
\tikzset{haha/.style={decoration={markings,mark=at position 1 with %
    {\arrow[scale=2,>=to]{>}}},postaction={decorate}}}
\tikzmath{\x1 = 3; \x2 = 0.3; \x3=2; \x4=1.2; \x5=10; \x6=90; \x7 = 4; \y1=3; \y2=5;}
\draw (0,0) ellipse (\x1 and \x1);
\draw[thick] (0:\x1-\x2) -- (0:\x1+\x2);
\draw[thick] (60:\x1-\x2) -- (60:\x1+\x2);
\draw[thick] (-60:\x1-\x2) -- (-60:\x1+\x2);
\draw[thick] (120:\x1-\x2) -- (120:\x1+\x2);
\draw[thick] (240:\x1-\x2) -- (240:\x1+\x2);

\node at (0:\x1+\x3*\x2) {$\xi_0$};
\node at (60:\x1+\x3*\x2) {$\xi_1$};
\node at (120:\x1+\x3*\x2) {$\xi_2$};
\node at (-60:\x1+\x3*\x2) {$\xi_k$};
\node at (-120:\x1+\x3*\x2) {$\xi_{k-1}$};

\node at (30:\x1-\x4*\x2) {$L_0$};
\node at (90:\x1-\x4*\x2) {$L_1$};
\node at (-30:\x1-\x4*\x2) {$L_k$};
\node at (-90:\x1-\x4*\x2) {$L_{k-1}$};

\node at (0,0) {\Huge $u$};

\draw[black,fill=black]   (150:\x1+\x3*\x2)  circle (.1ex) ;
\draw[black,fill=black]   (165:\x1+\x3*\x2)  circle (.1ex) ;
\draw[black,fill=black]   (180:\x1+\x3*\x2)  circle (.1ex) ;
\draw[black,fill=black]   (195:\x1+\x3*\x2)  circle (.1ex) ;
\draw[black,fill=black]   (210:\x1+\x3*\x2)  circle (.1ex) ;

\draw[thick, haha] (\x5:\x1+\x3*\x2) arc (\x5:60-\x5:\x1+\x3*\x2);
\draw[thick, haha] (60+\x5:\x1+\x3*\x2) arc (60+\x5:120-\x5:\x1+\x3*\x2);
\draw[thick, haha] (-120+\x5:\x1+\x3*\x2) arc (-120+\x5:-60-\x5:\x1+\x3*\x2);
\draw[thick, haha] (-60+\x5:\x1+\x3*\x2) arc (-60+\x5:-\x5:\x1+\x3*\x2);

\draw[thick] (60-\x5:\x1+\x3*\x2) to[out=60-\x5+90-\x6,in=60-90] (60:\x1+\x7*\x2);
\draw[thick, haha] (60:\x1+\x7*\x2) to[out=60+90,in=60+\x5-90+\x6] (60+\x5:\x1+\x3*\x2);

\draw[thick] (120-\x5:\x1+\x3*\x2) to[out=120-\x5+90-\x6,in=120-90] (120:\x1+\x7*\x2);
\draw[thick, haha] (120:\x1+\x7*\x2) to[out=120+90,in=120+\x5-90+\x6] (120+\x5:\x1+\x3*\x2);

\draw[thick] (-120-\x5:\x1+\x3*\x2)  to[out=-120-\x5+90-\x6,in=-120-90] (-120:\x1+\x7*\x2);
\draw[thick, haha] (-120:\x1+\x7*\x2) to[out=-120+90,in=-120+\x5-90+\x6] (-120+\x5:\x1+\x3*\x2);

\draw[thick] (-60-\x5:\x1+\x3*\x2) to[out=-60-\x5+90-\x6,in=-60-90] (-60:\x1+\x7*\x2);
\draw[thick, haha] (-60:\x1+\x7*\x2) to[out=-60+90,in=-60+\x5-90+\x6] (-60+\x5:\x1+\x3*\x2);

\draw[thick] (120+\x5:\x1+\x3*\x2) arc (120+\x5:140:\x1+\x3*\x2);
\draw[thick, haha] (-140:\x1+\x3*\x2) arc (-140:-120-\x5:\x1+\x3*\x2);

\node at (30:\x1+\y1*\x2) {\qquad \qquad  \tiny $~PT(\tilde{u}_{[\xi_0,\xi_1]},\ls_{L_0})$};
\node at (90:\x1+\y1*\x2) { \tiny  $PT(\tilde{u}_{[\xi_1,\xi_2]},\ls_{L_1})$};
\node at (-90:\x1+\y1*\x2) { \tiny  $PT(\tilde{u}_{[\xi_{k-1},\xi_k]},\ls_{L_{k-1}})$};
\node at (-30:\x1+\y1*\x2) {\qquad \qquad  \tiny  $~PT(\tilde{u}_{[\xi_k,\xi_0]},\ls_{L_k})$};

\node at (60:\x1+\y2*\x2) {$s_{C_1}$};
\node at (120:\x1+\y2*\x2) {$s_{C_2}$};
\node at (-60:\x1+\y2*\x2) {$s_{C_k}$};
\node at (-120:\x1+\y2*\x2) {$s_{C_{k-1}}$};
\end{tikzpicture}
\end{center}

Hence it suffices to show that this expression is equal to $\zt^{\mu(\b) -\left(\text{dim}(C_0) -\frac{1}{2}\text{dim}(X)\right) } s_{C'_0}$.

Since the concatenation of two canonical short paths in $\LG_{C_0}$, from $TL_k/TC_0$ to $TL_0/TC_0$ and from $TL_0/TC_0$ back to $TL_k/TC_0$, has Maslov index $-\text{dim}\left( TC_0^{\perp\w}/TC_0\right)=\text{dim}(C_0)-\frac{1}{2}\text{dim}(X)$, we have $s_{C_0'}\circ s_{C_0} = \zt^{\text{dim}(C_0) -\frac{1}{2}\text{dim}(X)} \text{id}_{\ls_{L_0}}$. Hence the last claim is equivalent to 
\begin{equation}\label{2}
s_{C_0}\circ  PT(\tilde{u}_{[\xi_k,\xi_0]},\ls_{L_k})\circ \left(s_{C_k}\right)_{\tilde{u}(\xi_k)} \circ \cdots \circ \left(s_{C_1}\right)_{\tilde{u}(\xi_1)} \circ PT(\tilde{u}_{[\xi_0,\xi_1]},\ls_{L_0}) = \zt^{\mu(\b)} \text{id}_{\ls_{L_0}}.
\end{equation}

To show \eqref{2}, notice that the domain of $u$ is contractible\footnote{In case $u$ contains sphere bubbles, we still regard the domain of $u$ as the punctured disk by contracting a finite collection of loops. The argument which follows will also work because Maslov index is additive with respect to cut-and-paste operation.}, and hence the bundle $u^*\LG'_X$ has a fiberwise infinite cover $\LG''\ra u^*\LG'_X$, i.e. its deck transformation group is isomorphic to $\ZZ$.

Consider the loop $\eta$ in $u^*\LG_X$ which is the concatenation of following paths (in the given order)
\[\tt_{L_0}(\tilde{u}_{[\xi_0,\xi_1]}), \tt^{C_1}_t(\tilde{u}(\xi_1)), \ldots, \tt^{C_k}_t(\tilde{u}(\xi_k)), \tt_{L_k}(\tilde{u}_{[\xi_k,\xi_0]}), \tt^{C_0}_t(\tilde{u}(\xi_0)). \]

Then the lifts of $\eta$ in $u^*\LG_M'$ and in $\LG''$ are paths whose end points are related by some group elements $a\in \ZZ_{2n}$ and $b\in\ZZ$ respectively. It is easy to see that $\zt^a=\zt^b$. By definition, $\zt^{a}\text{id}_{\ls_{L_0}}$ is equal to the LHS of \eqref{2} and $b=\mu(\b)$. This shows \eqref{2} and hence completes the proof of Theorem \ref{thm}.

\begin{Rem} The construction of the $A_{\infty}$ structure $m_k$ is usually supplied with some algebraic tools. For example, the perturbation theory of Kuranishi spaces only gives ``$m_{k,\b}$ modulo $T^E$'' \cite{Fu} or an ``$A_{N,K}$ structure'' \cite{AJ,FOOO}. In order to enhance them to an $A_{\infty}$ structure, algebraic arguments such as the homological perturbation and the approximate $A_{\infty}$ Whitehead's theorem are used. It is straightforward to keep track of these arguments and show that $\ac:Fuk(X)\ra Fuk(X)_{(\zt)}$ is an $A_{\infty}$ functor. The key point is that $\ac_1$ is a (twisted) chain isomorphism which allows us to transport the data used in the construction of the $A_{\infty}$ structure, such as the inputs for the homological perturbation and the homotopy inverses given by the approximate $A_{\infty}$ Whitehead's theorem, from the source $Fuk(X)$ to the target $Fuk(X)_{(\zt)}$ of $\ac$. Since we are using different data for the source and the target, we should consider $\ac$ only well-defined up to $A_{\infty}$ homotopy. In particular, $\ac^{2n}$ is only $A_{\infty}$ homotopic to the identity.
\end{Rem}

\section{Action on \texorpdfstring{$({\cal M},W)$}{}} Define 
\begin{align*}
\L&:=\left\{ \left.\sum_{i=0}^{\infty} a_i T^{\l_i} \right| a_i\in\CC,~ \l_0\leqslant\l_1\leqslant\cdots,~\lim_{i\ra+\infty}\l_i=+\infty\right\}\\
\L_0&:=\left\{ \left.\sum_{i=0}^{\infty} a_i T^{\l_i} \right| a_i\in\CC,~ 0\leqslant\l_0\leqslant\l_1\leqslant\cdots,~\lim_{i\ra+\infty}\l_i=+\infty\right\}\\
\L_+&:=\left\{ \left.\sum_{i=0}^{\infty} a_i T^{\l_i} \right| a_i\in\CC,~ 0<\l_0\leqslant\l_1\leqslant\cdots,~\lim_{i\ra+\infty}\l_i=+\infty\right\}\\
U_{\L}&:=\CC\oplus\L_+.
\end{align*}
Define the valuation $val:\L\ra \RR\cup\{+\infty\}$ by $val(0):=+\infty$ and 
\[val\left(\sum_{i=0}^{\infty} a_i T^{\l_i}\right):=\l_{\min\{i|a_i\ne 0\}}. \]
Similarly, define $val:V\otimes_{\CC}\L\ra \RR$ for any $\CC$-vector space $V$.

We start by recalling the Fukaya's trick and the definition of ${\cal M}_{weak}(L)$. Details can be found in \cite{Futrick,Tu,Yuan}.
\subsection{The Fukaya's trick and \texorpdfstring{${\cal M}_{weak}(L)$}{}} \label{recall} 

Let $L$ be a compact, oriented, relatively spin Lagrangian of $X$. In what follows, unless otherwise specified, the coefficient ring of the de Rham complex $\O(L)$ and the cohomology $H^{\bigcdot}(L)$ is taken to be $\L_0$. In \cite{Futrick}, Fukaya constructed, for any $\w$-tame almost complex structure $J$ on $X$, a \textit{cyclic unital filtered} $A_{\infty}$ structure $m^J=(m^J_{k,\b})$ on $\O(L)$ which satisfies the open string analogue of the \textit{divisor axiom} originating from the closed string Gromov-Witten theory:
\begin{equation}\label{DA1}
\sum_{m_0+\cdots+m_k=m} m^J_{k+m,\b}(b^{\otimes m_0}, x_1,b^{\otimes m_1},\ldots,b^{\otimes m_{k-1}}, x_k,b^{\otimes m_k} ) = \dfrac{\langle \partial \b, b\rangle^m}{m!} m_{k,\b}^J(x_1,\ldots,x_k)
\end{equation}
for any $\b$ with $\overline{\cal M}_1(L,\b,J)\ne\emptyset$, $k,m\geqslant 0$ with $(k,m,\b)\ne (0,1,0)$; $b\in\Omega^1(L)$ and $x_1,\ldots,x_k\in \O(L)$.

Moreover, for any smooth family ${\cal J}=\{J_t\}_{t\in [0,1]}$ of $\w$-tame almost complex structures on $X$, the $A_{\infty}$ quasi-isomorphism $\fr^{\cal J}: (\O(L), m^{J_0}) \ra (\O(L), m^{J_1})$ induced by the associated pseudo-isotopy is cyclic, unital and satisfies the similar divisor axiom:
\begin{equation}\label{DA2}
\sum_{m_0+\cdots+m_k=m} \fr^{\cal J}_{k+m,\b}(b^{\otimes m_0},  x_1,b^{\otimes m_1},\ldots,b^{\otimes m_{k-1}}, x_k,b^{\otimes m_k} ) = \dfrac{\langle \partial \b, b\rangle^m}{m!} \fr^{\cal J}_{k,\b}(x_1,\ldots,x_k)
\end{equation}
for any $\b$ with $\overline{\cal M}_1(L,\b,J)\ne\emptyset$, $k,m\geqslant 0$ with $(k,m,\b)\ne (0,1,0)$; $b\in\Omega^1(L)$ and $x_1,\ldots,x_k\in \O(L)$.

Notice that $m^J_{1,0}$ is equal to the de Rham differential and $\fr^{\cal J}_{1,0}$ is equal to the identity.

By passing to the \textit{canonical model} via homological perturbation, the above story holds with $\O(L)$ replaced by $\H(L)$, except that $m^J_{1,0}$ is now equal to zero.

An immediate consequence of \eqref{DA1} is the following. Let $b\in H^1(L)$. Define $m^{J,b}=(m^{J,b}_{k,\b})$ where $m^{J,b}_{k,\b}:\H(L)^{\otimes k}\ra \H(L)$ is given by
\[ m^{J,b}_{k,\b}(x_1,\ldots,x_k):= \sum_{m=0}^{\infty}\sum_{m_0+\cdots+m_k=m} m^J_{k+m,\b}(b^{\otimes m_0},  x_1,b^{\otimes m_1},\ldots,b^{\otimes m_{k-1}}, x_k,b^{\otimes m_k}  ).\]
By \eqref{DA1}, $m^{J,b}_{k,\b}=e^{\langle\partial \b,b\rangle} m^J_{k,\b}$. In other words, $m^{J,b}$ is equal to $m^J$ twisted by the local system $\ls_b$ with holonomy $e^{\langle -,b\rangle}$. This allows us to identity the Lagrangian brane $(L,\ls_b)$ with the Lagrangian brane $(L,b)$ where the latter does not carry any local system but an element $b\in H^1(L)$ which will play the role of weak bounding cochain.

Take a basis $\{e_1,\ldots,e_{\ll}\}$ of $H^1(L;\ZZ)/\text{torsion}$. Then every element $b\in H^1(L)$ can be written uniquely as $b=x_1 e_1+\cdots + x_{\ll} e_{\ll}$ with $x_1,\ldots,x_{\ll}\in\L_0$. Put $y_i=e^{x_i}$. Then $y_1,\ldots, y_{\ll}$ are coordinates of the $\ll$-torus $H^1(L;\L^{\times})$ over the Novikov field\footnote{In order for $y_i=e^{x_i}$ to have meaning, it is necessary and sufficient that $y_i\in U_{\L}$. But for the purpose of extending the domain of definition of the function $P$ which will be introduced shortly, we regard each $y_i$ formally as an element of $\L^{\times}$.} ($\L^{\times}:=\L-\{0\}$).

Consider a formal power series $P$ on $H^1(L;\L^{\times})\oplus H^{\text{odd}>1}(L;\L)$ defined by 
\[P(y_1,\ldots,y_{\ll},b_{>1}):=\sum_{k,\b} T^{\energy{\b}} y_1^{\langle \partial \b, e_1\rangle} \cdots y_{\ll}^{\langle \partial \b, e_{\ll}\rangle} m^J_{k,\b}( (b_{>1})^{\otimes k}).\]

\begin{Th}\cite[Theorem 1.2]{Futrick}\label{Futrick} $P$ is convergent in 
\[{\cal V}_{\delta}=\{(y_1,\ldots,y_{\ll},b_{>1})|~val(y_1),\ldots,val(y_{\ll})\in (-\delta,\delta),~b_{>1}\in H^{\text{odd}>1}(L;\L_+)\}\]
where $\delta>0$ is a positive constant.
\end{Th}
The proof is based on the following argument which is what the Fukaya's trick refers to. Consider a Weinstein neighbourhood $U$ of $L$. Then every small $\a=v_1e_1+\cdots+v_{\ll}e_{\ll}\in H^1(L;\RR)$, i.e. $v_1,\ldots, v_{\ll}$ are real numbers close enough to $0$, gives rise to a nearby Lagrangian $L(\a)$ lying inside $U$ which is expressed as the graph of a closed 1-form representing $\a$. Take a diffeomorphism $\diifeo_{\a}$ of $X$ such that $\diifeo_{\a}(L)=L(\a)$ and $(\diifeo_{\a})_*J:= d\diifeo_{\a}\circ J\circ (d\diifeo_{\a})^{-1}$ is $\w$-tame\footnote{The latter condition is satisfied if $\a$ is small enough.}. Write $\b':=(\diifeo_{\a})_*\b$, $e'_i=(\diifeo_{\a})_*e_i$, etc. Let $E(\b)$ denote the symplectic area $\int_{\b}\w$.

By the facts that $\energy{\b'}=\energy{\b}+\langle \partial\b,\a\rangle$ and that the moduli spaces $\overline{{\cal M}}_{k+1}(L,\b,J)$ and $\overline{{\cal M}}_{k+1}(L(\a),\b',(\diifeo_{\a})_*J)$ are identical as \textit{Kuranishi spaces} (see \cite{AJ, Futrick, FO, FOOO} for the definition), we have
\begin{align*}
P(y_1,\ldots,y_{\ll},b_{>1}) &= \sum_{k,\b'} T^{E(\b')} \cdot T^{-\langle\partial\b, v_1e_1+\cdots+v_{\ll}e_{\ll} \rangle}y_1^{\langle \partial \b', e'_1\rangle} \cdots y_{\ll}^{\langle \partial \b', e'_{\ll}\rangle}  m^{(\diifeo_{\a})_*J}_{k,\b'}( (b'_{>1})^{\otimes k})\\
&= \sum_{k,\b'} T^{E(\b')} \left(T^{-v_1}y_1\right)^{\langle \partial \b', e'_1\rangle} \cdots \left(T^{-v_{\ll}}y_{\ll}\right)^{\langle \partial \b', e'_{\ll}\rangle}  m^{(\diifeo_{\a})_*J}_{k,\b'}( (b'_{>1})^{\otimes k}).
\end{align*}
It follows that $P(y_1,\ldots,y_{\ll},b_{>1}) $ is convergent if $val(y_i)=v_i,~i=1,\ldots,\ll$, by Gromov compactness. 

\begin{Def} Define the local mirror of $L$ (with respect to $J$) to be
\[  {\cal M}_{weak}(L,m^J):=\{  \mb{b}=(y_1,\ldots,y_{\ll},b_{>1})\in{\cal V}_{\delta}|~P(\mb{b})=W(\mb{b})\cdot 1\text{ for some scalar }W(\mb{b})\}/_{\sim} \]
where $\sim$ is the gauge equivalence \cite{FOOO}.
\end{Def}

The proof of Theorem \ref{Futrick} suggests that ${\cal M}_{weak}(L,m^J)$ contains not only the weak bounding cochains (over $\L_+$) on $L$ but also those on all nearby Lagrangians (up to Hamiltonian isotopy). Therefore, for any two Lagrangians $L$ and $L'$ which are close to each other, their local mirrors overlap. More precisely, they contain weak bounding cochains (over $\L_+$) on all Lagrangians which are close to $L$ and $L'$ simultaneously. The gluing function defined on this overlapping region can be obtained as follows. 

Take a path $\{\diifeo^t_{L,L'}\}_{t\in [0,1]}$ of diffeomorphisms of $X$ such that $\diifeo^0_{L,L'}=\text{id}$, $\diifeo^1_{L,L'}(L)=L'$ and $(\diifeo^t_{L,L'})^{-1}_*J$ remains $\w$-tame for all $t$. The family ${\cal J}:= \{(\diifeo^t_{L,L'})^{-1}_*J\}_{t\in [0,1]}$ of $\w$-tame almost complex structures induces an $A_{\infty}$ quasi-isomorphism $\fr^{\cal J}$ satisfying \eqref{DA2}. Define 
\[\fr^{\cal J}_*: {\cal M}_{weak}(L,m^J) \dasharrow {\cal M}_{weak}(L, m^{ (\diifeo^1_{L,L'})^{-1}_*J})\]
by 
\[\fr^{\cal J}_*(\mb{b}) :=\left( y_1 e^{\langle \pr_1(f(\mb{b})), e_1^{\vee}\rangle},\ldots, y_{\ll} e^{\langle \pr_1(f(\mb{b})), e_{\ll}^{\vee}\rangle}, \pr_{\ne 1}(f(\mb{b}))\right) \]
where $\{e_1^{\vee},\ldots,e_{\ll}^{\vee}\}$ is the dual basis of $\{e_1,\ldots,e_{\ll}\}$, $\pr_1$ (resp. $\pr_{\ne 1}$) is the projection of $H^{\bigcdot}(L;\L)$ onto $H^1(L;\L)$ (resp. $\bigoplus_{d\ne 1}H^d(L;\L)$), and $f(\mb{b})=\sum_{k,\b}T^{\energy{\b}} y_1^{\langle \partial\b,e_1\rangle}\cdots y_{\ll}^{\langle \partial \b,e_{\ll}\rangle} \fr^{\cal J}_{k,\b}((b_{>1})^{\otimes k})$. Here the dash arrow indicates that $\fr_*^{\cal J}$ is defined only on an open subset of the domain due to the convergence issue which will be discussed shortly. 

On the other hand, the diffeomorphism $\diifeo^1_{L,L'}$ induces a bijection $\psi:  {\cal M}_{weak}(L,m^{(\diifeo^1_{L,L'})^{-1}_*J} ) \ra  {\cal M}_{weak}(L',m^J)$ given by 
\[ \psi(y_1,\ldots, y_{\ll}, b_{>1}) := \left( T^{-v_1} (\diifeo^1_{L,L'})_*y_1,\ldots,T^{-v_{\ll}} (\diifeo^1_{L,L'})_*y_{\ll} , (\diifeo^1_{L,L'})_*(b_{>1})\right)\]
where we regard $L'$ as the graph of a closed 1-form representing $v_1e_1+\cdots+ v_{\ll}e_{\ll}$. 

Then the desired gluing function $\Psi_{L,L'}$ is defined to be the composite
\[{\cal M}_{weak}(L,m^J) \xdashrightarrow{$~\fr^{\cal J}_*~$ }{} {\cal M}_{weak}(L,m^{ (\diifeo^1_{L,L'})^{-1}_*J})\xrightarrow{~\psi~} {\cal M}_{weak}(L',m^J). \]
It turns out that $\fr^{\cal J}_*$ is convergent at any point which corresponds to a third Lagrangian $L''$ which is close to $L$ and $L'$ simultaneously (i.e. the overlapping region). This is proved by a family version of the Fukaya's trick \cite{Tu}. The key point is to find a family $\{{\cal G}_t\}_{t\in [0,1]}$ of diffeomorphisms such that ${\cal G}_t(L)=L''$ for all $t$ and $\{({\cal G}_t)_*{\cal J}\}_{t\in [0,1]}$ is a family of $\w$-tame almost complex structures joining $(\diifeo_{\a})_*J$ and $(\diifeo_{\a'})_*J$ where $L''=\diifeo_{\a}(L)=\diifeo_{\a'}(L')$ and $\diifeo_{\a}$ (resp. $\diifeo_{\a'}$) are the diffeomorphisms used in the proof of Theorem \ref{Futrick} for $L$ (resp. $L'$). 

\subsection{Action on \texorpdfstring{$({\cal M}_{weak}(L),W)$}{}} Recall the $\ZZ_{2n}$-local system $\ls_L$ from Section \ref{object level}. Since $L$ is oriented, $\ls_L$ is reduced to a $\ZZ_n$-local system. See Remark \ref{order=n}. Hence $\ls_L$ is represented by an element $\g\in Hom(H_1(L;\ZZ),\CC^{\times})$ with the property that 
\[\g(\partial\b)=\zt^{\mu(\b)}\]
for any $\b\in \pi_2(X,L)$ where $\zt=e^{\frac{2\pi i}{2n}}$. See \eqref{2}. Write $\g_i:=\g(e_i^{\vee})$, $i=1,\ldots,\ll$.

Apply the operator $\tww$ from Section \ref{mor level} to $\text{id}_{H^{\bigcdot}(L;\L_+)}$. Recall it means
\[\tw{\text{id}}|_{H^d(L;\L_+)}=\zt^{-d}\text{id}_{H^d(L;\L_+)}.\]

\begin{Def} Define $\tau:{\cal V}_{\delta}\ra {\cal V}_{\delta}$ by
\[\tau(y_1,\ldots,y_{\ll},b_{>1}):= \left( \g_1y_1,\ldots, \g_{\ll} y_{\ll},\zt \tw{\text{id}}(b_{>1}) \right).\]
\end{Def}

\begin{Lemma} \label{actiononlocalmirror} If $\mb{b}\in {\cal M}_{weak}(L,m^J)$, then $\tau(\mb{b})\in {\cal M}_{weak}(L,m^J)$ and $W(\tau(\mb{b}))=\zt^2 W(\mb{b})$.
\end{Lemma}
\begin{proof}
Write $b_{>1}=\sum b_i$ with $b_i\in H^i(L;\L_+)$. We have 
\begin{align*}
P(\mb{b})&=\sum_{k,\b}\sum_{i_1,\ldots, i_k} T^{\energy{\b}} y_1^{\langle \partial \b, e_1\rangle} \cdots y_{\ll}^{\langle \partial \b, e_{\ll}\rangle} m^J_{k,\b}(b_{i_1},\ldots, b_{i_{\ll}})\\
&= W(\mb{b})\cdot 1.
\end{align*}
On the other hand, 
\begin{align*}
&~P(\tau(\mb{b}))\\
=&\sum_{k,\b}\sum_{i_1,\ldots, i_k} T^{\energy{\b}} \left(\g_1 y_1\right)^{\langle \partial \b, e_1\rangle} \cdots \left(\g_{\ll}y_{\ll}\right)^{\langle \partial \b, e_{\ll}\rangle} \zt^{k-(i_1+\cdots +i_k)} m^J_{k,\b}(b_{i_1},\ldots, b_{i_{\ll}})\\
=&\sum_{k,\b}\sum_{i_1,\ldots, i_k} T^{\energy{\b}}  y_1^{\langle \partial \b, e_1\rangle} \cdots y_{\ll}^{\langle \partial \b, e_{\ll}\rangle} \g(\langle \partial \b,e_1\rangle e_1^{\vee}+\cdots+ \langle \partial\b,e_{\ll}\rangle e_{\ll}^{\vee} )\zt^{k-(i_1+\cdots+ i_k)} m^J_{k,\b}(b_{i_1},\ldots, b_{i_{\ll}})\\
=& \sum_{k,\b}\sum_{i_1,\ldots, i_k} T^{\energy{\b}}  y_1^{\langle \partial \b, e_1\rangle} \cdots y_{\ll}^{\langle \partial \b, e_{\ll}\rangle} \zt^{\mu(\b)+k-(i_1+\cdots+ i_k)} m^J_{k,\b}(b_{i_1},\ldots, b_{i_{\ll}})\\
=&~ \zt^2~ \sum_{k,\b}\sum_{i_1,\ldots, i_k} T^{\energy{\b}}  y_1^{\langle \partial \b, e_1\rangle} \cdots y_{\ll}^{\langle \partial \b, e_{\ll}\rangle} \zt^{-[i_1+\cdots +i_k+2-k-\mu(\b)]} m^J_{k,\b}(b_{i_1},\ldots, b_{i_{\ll}})\\
\end{align*}
Notice that $m^J_{k,\b}(b_{i_1},\ldots, b_{i_k})$ has degree $i_1+\cdots+i_k+2-k-\mu(\b)$, and hence 
\[P(\tau(\mb{b}))=\zt^2~\tw{\text{id}}(P(\mb{b}))=\zt^2~\tw{\text{id}}(W(\mb{b})\cdot 1)=\zt^2W(\mb{b}).\]
The proof that $\tau$ preserves gauge equivalences is similar.
\end{proof}
It is clear that $\tau^n=\text{id}$. We have proved
\begin{Prop}(=Proposition \ref{localaction}) \label{localaction3} There is a $\ZZ_n$-action on ${\cal M}_{weak}(L,m^J)$ with respect to which $W$ is equivariant.
\end{Prop} 
\subsection{Commute with wall-crossing} Recall $\Psi_{L,L'}$ is defined to be the composite
\[{\cal M}_{weak}(L,m^J) \xdashrightarrow{$~\fr^{\cal J}_*~$ }{} {\cal M}_{weak}(L,m^{ (\diifeo^1_{L,L'})^{-1}_*J})\xrightarrow{~\psi~} {\cal M}_{weak}(L',m^J). \]

By Proposition \ref{localaction3}, $\ZZ_n$ acts on the local mirrors ${\cal M}_{weak}(L,m^J)$, $ {\cal M}_{weak}(L,m^{ (\diifeo^1_{L,L'})^{-1}_*J})$ and ${\cal M}_{weak}(L',m^J)$ with respect to which $W$ is equivariant. 

\begin{Lemma} $\tau$ commutes with $\fr^{\cal J}_*$. 
\end{Lemma}
\begin{proof} Write $b_{>1}=\sum b_i$ as before. We have
\begin{align*}
f(\tau(\mb{b}))&= \sum_{k,\b}T^{\energy{\b}} \left(\g_1 y_1\right)^{\langle \partial \b, e_1\rangle} \cdots \left(\g_{\ll}y_{\ll}\right)^{\langle \partial \b, e_{\ll}\rangle} \fr^{\cal J}_{k,\b}( (\zt~\tw{\text{id}}(b_{>1}))^{\otimes k})\\
&= \sum_{k,\b}\sum_{i_1,\ldots, i_k} T^{\energy{\b}} y_1^{\langle \partial \b, e_1\rangle} \cdots y_{\ll}^{\langle \partial \b, e_{\ll}\rangle} \zt^{\mu(\b)+k-(i_1+\cdots+i_k)}\fr^{\cal J}_{k,\b}(b_{i_1},\ldots,b_{i_k})\\
&= \zt~  \sum_{k,\b}\sum_{i_1,\ldots, i_k} T^{\energy{\b}} y_1^{\langle \partial \b, e_1\rangle} \cdots y_{\ll}^{\langle \partial \b, e_{\ll}\rangle} \zt^{-[ i_1+\cdots+i_k+1-k-\mu(\b) ]}\fr^{\cal J}_{k,\b}(b_{i_1},\ldots,b_{i_k}).\\
\end{align*}
Since $\fr^{\cal J}_{k,\b}(b_{i_1},\ldots,b_{i_k})$ has degree $i_1+\cdots+i_k+1-k-\mu(\b)$, we have 
\[f(\tau(\mb{b}))=\zt~\tw{\text{id}}(f(\mb{b})).\]

Then 
\begin{align*}
\fr^{\cal J}_*\circ \tau(\mb{b}) &= \left( \g_1 y_1 e^{\langle\pr_1(f(\tau(\mb{b}))),e_1^{\vee} \rangle} , \ldots, \g_{\ll} y_{\ll} e^{\langle\pr_1(f(\tau(\mb{b}))),e_{\ll}^{\vee} \rangle}, \pr_{\ne 1}(f(\tau(\mb{b})))\right)\\
&=\left( \g_1 y_1 e^{\langle\pr_1(f(\mb{b})),e_1^{\vee} \rangle} , \ldots, \g_{\ll} y_{\ll} e^{\langle\pr_1(f(\mb{b})),e_{\ll}^{\vee} \rangle}, \zt~\tw{\text{id}}(\pr_{\ne 1}(f(\mb{b})))\right)\\
&=\tau\circ \fr^{\cal J}_*(\mb{b}).
\end{align*}
(We have used the fact that $\tw{\text{id}}$ commutes with $\pr_1$ and $\pr_{\ne 1}$.)
\end{proof}

\begin{Lemma} $\tau$ commutes with $\psi$.
\end{Lemma}
\begin{proof}
It follows from the fact that if $\iota_t:L\ra X$ is a Lagrangian isotopy, then the local systems $\iota_0^*\ls_{\iota_0(L)}$ and $\iota_1^*\ls_{\iota_1(L)}$ on $L$ are isomorphic, as they are isomorphic to a local system on $L\times [0,1]$ restricted to the slices $L\times \{0\}$ and $L\times \{1\}$ respectively. 
\end{proof}

Therefore, we have proved 
\begin{Prop}(=Proposition \ref{commutewithwallcrossing}) $\tau$ commutes with $\Psi_{L,L'}$.
\end{Prop}

\section{Extending the action to the complete mirror}
Let $X,D$ and $(\check{X}^{\circ},W)$ be given as in the introduction. Recall that $X$ is K\"ahler, $D$ is an anticanonical divisor of $X$ and $(\check{X}^{\circ},W)$ is the uncompactified mirror obtained by gluing the local mirror charts of the smooth torus fibers of an SYZ fibration on $X-D$ following Fukaya's scheme. 

\noindent\textbf{Assumptions (B).}
\begin{itemize}
\item $\check{X}^{\circ}$ is an analytic variety over $\CC$.
\item $(\check{X}^{\circ},W)$ can be completed to $(\check{X},W)$ where $\check{X}$ is a normal affine analytic variety and $W:\check{X}\ra\CC$ is a holomorphic function. 
\item The complement $\check{X}-\check{X}^{\circ}$ is contained in a closed analytic subset $A$ which has codimension at least two. 
\item There is a $\ZZ_n$-action on $(\check{X}^{\circ},W)$, i.e. there is a biholomorphism $\tau:\check{X}^{\circ}\ra \check{X}^{\circ}$ such that $\tau^n=\text{id}$ and 
\[W(\tau\cdot x)=e^{\frac{2\pi i}{n}} W(x),~x\in \check{X}^{\circ}.\]
\end{itemize}
Notice that the last assumption is actually the outcome of Corollary \ref{cor}. As for the first three, we emphasize that they are reasonable if $X$ is Fano. For example, consider the complete SYZ mirror of the famous special Lagrangian torus fibration defined on $\CC P^2$ minus a line and a conic which is given by 
\begin{align*}
\check{X}&=\{(u,v)\in\CC^2|~uv\ne 1\}\\
W&= u+\frac{v^2}{uv-1}.
\end{align*}
The only point that the local mirror charts of the Clifford tori and Chekanov tori do not cover is $(0,0)$. (It is covered by the local mirror chart of the immersed 2-sphere.) This point has codimension two, and the desired action is given by $(u,v)\mapsto (\zt u, \zt^{-1} v)$ where $\zt=e^{\frac{2\pi i}{3}}$.

Back to the general case. 
\begin{Prop} Under the assumptions (B), the $\ZZ_n$-action on $(\check{X}^{\circ},W)$ extends to a unique $\ZZ_n$-action on $(\check{X},W)$.
\end{Prop}
\begin{proof}
Let $U:=\check{X}-A$. Define
\[V:=\bigcap_{i=0}^{n-1} \tau^i(U).\]
Then $V$ is an open analytic subset of $\check{X}$ whose complement has codimension at least two. Moreover, $V$ is contained in $\check{X}^{\circ}$ and is invariant under the given $\ZZ_n$-action. It follows that $\ZZ_n$ acts on the ring $\CC[V]$ of holomorphic functions on $V$. But the inclusion $V\hookrightarrow\check{X}$ induces an isomorphism $\CC[\check{X}]\simeq\CC[V]$ by
\begin{Lemma} (The second Riemann extension theorem, see e.g. \cite[Chapter 7]{GR}) Every holomorphic function on $V$ extends to a unique holomorphic function on $\check{X}$. 
\end{Lemma}
It follows that $\ZZ_n$ acts on the ring $\CC[\check{X}]$ and hence on the space $\check{X}$, since $\check{X}$ is affine. It is clear that this action is the unique extension of the given one on $\check{X}^{\circ}$ and $W$ is equivariant in the above sense.
\end{proof}
\appendix
\section{\texorpdfstring{$Fuk(X)$}{}}\label{Fuk(X)}
The objects of $Fuk(X)$ consist of $\LL=(L,\ls)$\footnote{A relative spin structure on $L$ is also included as part of the data. But since it is not relevant for our construction, we drop it from the discussion.} where $L$ is taken from a fixed finite collection of immersed compact oriented Lagrangians of $X$ with clean self-intersection and $\ls$ is a $\CC^{\times}$-local system on $L$.

For any pair $L_0, L_1$ of such Lagrangians, let $\iota: L_i\ra X,~i=0,1$ denote the immersion. Suppose $L_0$ and $L_1$ intersect cleanly. Recall it means the fiber product
\[L_0\times_{\iota} L_1:=\{ (x,y)\in L_0\times L_1|~\iota(x)=\iota(y)\}\]
is smooth and satisfies 
\[ T_{(x,y)}(L_0\times_{\iota} L_1)=T_xL_0\times_{d\iota} T_y L_1\]
for any $(x,y)\in L_0\times_{\iota} L_1$.

Define the morphism space between two objects $\LL_i=(L_i,\ls_i),~i=0,1$ by
\begin{equation}\label{mor}
Hom_{Fuk(X)}(\LL_0,\LL_1):=\bigoplus_{C\in \pi_0(L_0\times_{\iota} L_1)} \O(C;Hom(\ls_0|_C,\ls_1|_C))
\end{equation}
where $\O(C;\ls)$ is any of the standard models (de Rham, singular cochain, etc) of the cohomology group $H^{\bigcdot}(C;\ls)$ with local coefficient $\ls$\footnote{Strictly speaking, the local system $Hom(\ls_0|_C,\ls_1|_C)$ in \eqref{mor} has to be twisted by a $\ZZ_2$-local system which is used to orient the moduli spaces of holomorphic disks. Since our $\ac$ will not modify it, we drop it from the discussion.}.

The $A_{\infty}$ structure on $Fuk(X)$ is defined by making sense of the slogan ``counting holomorphic polygons''. Let $\overrightarrow{\LL}=(\LL_0,\ldots,\LL_k)$ be a sequence of $k+1$ objects of $Fuk(X)$ such that $L_{i-1}$ and $L_i$ intersect cleanly for each $i=0,\ldots,k$. (Here $\LL_{-1}=\LL_k$.) Let $\overrightarrow{L}=(L_0,\ldots,L_k)$. For each $i=0,\ldots, k$, fix a connected component $C_i\in\pi_0(L_{i-1}\times_{\iota} L_i)$. Let $C'_0$ denote the connected component $C_0$ regarded as an element of $\pi_0(L_0\times_{\iota}L_k)$. Fix a homotopy class $\b\in \pi_2\left(X, \bigcup_{i=0}^k\iota(L_i)\right)$ and an $\w$-tame almost complex structure $J$ on $X$. 

The domain $\Sigma$ of holomorphic polygons which will be considered are bordered Riemann surfaces of genus zero with boundary marked points $\xi_0,\ldots,\xi_k,\xi_{k+1}=\xi_0$ arranged counterclockwise. For each pair of consecutive marked points $\xi_i$ and $\xi_{i+1}$, we denote by $[\xi_i,\xi_{i+1}]$ the arc in $\partial\Sigma$ drawn from $\xi_i$ to $\xi_{i+1}$ counterclockwise.

Define $\Moduli$ to be the moduli space of $J$-holomorphic polygons $u$ which represents $\b$ and has continuous lifts $\tilde{u}_{[\xi_i,\xi_{i+1}]}, i=0,\ldots, k$ into $L_i$ along the arc $[\xi_i,\xi_{i+1}]$ such that for each $i$
\[\tilde{u}(\xi_i) := (\tilde{u}(\xi_i^-),\tilde{u}(\xi_i^+)) := \left( \lim_{\substack{\xi\ra\xi_i\\\xi\in(\xi_{i-1},\xi_i) }} \tilde{u}_{[\xi_{i-1},\xi_i]}(\xi) , \lim_{\substack{\xi\ra\xi_i\\\xi\in(\xi_i,\xi_{i+1}) }} \tilde{u}_{[\xi_i,\xi_{i+1}]}(\xi)   \right)\in C_i.\] 

The virtual dimension of $\Moduli$ is equal to $\frac{1}{2}\text{dim}(X)+\mu(\b)+k-2$ where $\mu(\b)$ is the Maslov index of $\b$ which is defined in the standard way. See \cite{AJ}. Notice that in the presence of corners, $\mu(\b)$ depends on an assignment to each marked point $\zt_i$ a path in the Lagrangian Grassmannian $LG(T_{u(\xi_i)}X,\w_{u(\xi_i)} )$ joining the two limiting Lagrangian subspaces at $u(\xi_i)$ (from the left and from the right) determined by a representative $u$. In our case, we have chosen the canonical short path \eqref{Lagpath} from Section \ref{mor level}.

By performing abstract or rigid count of elements of $\Moduli$, weighted by the holonomy of $\ls_i$'s along their boundary arcs, one obtains a multilinear map
\begin{equation}\label{m_kb}
m_{k,\b,\overrightarrow{\LL},\overrightarrow{C}}: \bigotimes_{i=1}^k \O(C_i; Hom(\ls_{i-1}|_{C_i} ,\ls_i|_{C_i}  ) )\ra \O(C'_0; Hom(\ls_0|_{C'_0} ,\ls_k|_{C'_0}  ) ) 
\end{equation}
of degree $\left(\text{dim}(C_0)-\frac{1}{2}\text{dim}(X)\right)+2-k-\mu(\b)$. (The degree can be seen from the dimension of $\Moduli$ given above.)

The $A_{\infty}$ product map 
\[m_k:\bigotimes_{i=1}^k Hom_{Fuk(X)}(\LL_{i-1},\LL_i)\ra Hom_{Fuk(X)}(\LL_0,\LL_k)\]
is then defined to be 
\begin{equation}\label{m_k}
m_k:=\sum_{\overrightarrow{C},\b} T^{\energy{\b}} m_{k,\b,\overrightarrow{\LL},\overrightarrow{C}}
\end{equation}
over all $\overrightarrow{C}=(C_0,\ldots,C_k)$ with $C_i\in \pi_0(L_{i-1}\times_{\iota}L_i)$ and $\b\in\pi_2\left(X, \bigcup_{i=0}^k\iota(L_i)\right)$ for which $\Moduli\ne \emptyset$. (Here $E(\b)=\int_{\b}\w$ is the symplectic area.)

In general, $Fuk(X)$ is defined over the Novikov ring $\L_0$.

\setstretch{1.2}


\begin{thebibliography}{99} 
\bibitem{Ab1}
M. Abouzaid, `Family Floer cohomology and mirror symmetry', 2014, \href{https://arxiv.org/abs/1404.2659}{arXiv:1404.2659}.

\bibitem{Ab2}
M. Abouzaid, `The family Floer functor is faithful', \textit{J. Eur. Math. Soc.} \textbf{19}(7) (2017) 2139-2217.

\bibitem{Ab3}
M. Abouzaid, `Homological mirror symmetry without corrections', 2017, \href{https://arxiv.org/abs/1703.07898}{arXiv:1703.07898}.

\bibitem{AJ} A. Akaho and D. Joyce, `Immersed Lagrangian Floer theory',  \textit{J. Differential Geom.} \textbf{86}(3) (2010), 381-500.

\bibitem{Auroux2} D. Auroux, `Mirror symmetry and T-duality in the complement of an anticanonical divisor',  \textit{J. G\"okova Geom. Topol. GGT} \textbf{1} (2007), 51-91. 

\bibitem{Auroux} D. Auroux, `A beginner's introduction to Fukaya categories', in \textit{Contact and symplectic topology}, Bolyai Soc. Math. Stud., 26 (J\'anos Bolyai Mathematical Society, Budapest, 2014), 85-136. 

\bibitem{AS} D. Auroux and I. Smith, `Fukaya categories of surfaces, spherical objects, and mapping class groups', 2020, \href{https://arxiv.org/abs/2006.09689}{arXiv:2006.09689}

\bibitem{CLL} K. Chan, S.-C. Lau and N.-C. Leung, `SYZ mirror symmetry for toric Calabi-Yau manifolds',  \textit{J. Differential Geom.} \textbf{90}(2) (2012), 177-250.

\bibitem{CH} C. Cho and H. Hong, `Finite group actions on Lagrangian Floer theory', \textit{J. Symplectic Geom.} \textbf{15}(2) (2017), 307-420.

\bibitem{CHL1} C. Cho, H. Hong and S.-C. Lau, `Gluing localized mirror functors', 2018, \href{https://arxiv.org/abs/1810.02045}{arXiv:1810.02045}. 

\bibitem{CHL3} \underline{\qquad\qquad}, `Localized mirror functor for Lagrangian immersions, and homological mirror symmetry for $\mathbb{P}^1_{a,b,c}$',  \textit{J. Differential Geom.} \textbf{106}(1) (2017), 45-126.

\bibitem{CO} C. Cho and Y. Oh, `Floer cohomology and disc instantons of Lagrangian torus fibers in Fano toric manifolds'.  \textit{Asian J. Math} \textbf{10}(4) (2006), 773-814. 

\bibitem{Futrick}
K. Fukaya, `Cyclic symmetry and adic convergence in Lagrangian Floer theory', \textit{Kyoto J. Math.} \textbf{50}(3) (2010), 521-590.

\bibitem{FuGalois} K. Fukaya, `Galois symmetry on Floer cohomology'.  \textit{Turkish J. Math.} \textbf{27}(1) (2003), 11-32.

\bibitem{Fu} K. Fukaya, `Unobstructed immersed Lagrangian correspondence and filtered A infinity functor', 2017, \href{https://arxiv.org/abs/1706.02131}{arXiv:1706.02131}. 

\bibitem{FO} 
K. Fukaya and K. Ono, `Arnold conjecture and Gromov-Witten invariant',  \textit{Topology} \textbf{38} (1999), 933-1048.

\bibitem{FOOO} K. Fukaya, Y. Oh, H. Ohta and K. Ono,  \textit{Lagrangian intersection Floer theory - anomaly and obstruction. Part I and II}, AMS/IP Stud. Adv. Math., 46 (American Mathematical Society, Providence, RI, 2009).

\bibitem{FOOOAn} \underline{\qquad\qquad}, `Anchored Lagrangian submanifolds and their Floer theory', in \textit{Mirror symmetry and tropical geometry}, Contemp. Math., 527, (American Mathematical Society, Providence, RI, 2010), 15-54. 

\bibitem{FOOO toric} \underline{\qquad\qquad}, `Lagrangian Floer theory on compact toric manifolds. I', \textit{Duke Math. J.} \textbf{151}(1) (2010), 23-174.

\bibitem{GR}
H. Grauert, R. Remmert, \textit{Coherent analytic sheaves}, Grundlehren der Mathematischen Wissenschaften, 265 (Springer-Verlag, Berlin, 1984).

\bibitem{HKL} H. Hong, Y. Kim and S.-C. Lau, `Immersed two-spheres and SYZ with application to Grassmannians', 2018, \href{https://arxiv.org/abs/1805.11738}{arXiv:1805.11738}.

\bibitem{Kont} M. Kontsevich, `Homological algebra of mirror symmetry', in \textit{Proc. Int. Cong. Math. (Z\"urich, 1994)} \textbf{1} (Birkh\"auser, Basel, 1995), 120-139. 

\bibitem{KS} M. Kontsevich and Y. Soibelman, `Affine structures and non-Archimedean analytic spaces', in \textit{The unity of mathematics}, Progr. Math., 244 (Birkh\"auser Boston, Boston, MA, 2006), 321-385.

\bibitem{Ku} A. Kuznetsov and M. Smirnov, `On residual categories for Grassmannians', \textit{Proc. Lond. Math. Soc. (3)}, to appear. Preprint, 2018, \href{https://arxiv.org/abs/1802.08097}{arXiv:1802.08097}.

\bibitem{Ku2} A. Kuznetsov and M. Smirnov, `Residual categories for (co)adjoint Grassmannians in classical types', 2020, \href{https://arxiv.org/abs/2001.04148}{arXiv:2001.04148}.

\bibitem{Oh} Y. Oh, `Mean curvature vector and symplectic topology of Lagrangian submanifolds in Einstein-K\"ahler manifolds', \textit{Math. Z.}, \textbf{216} (1994), 471-482.

\bibitem{Pascal} J. Pascaleff and D. Tonkonog, `The wall-crossing formula and Lagrangian mutations', 2017, \href{https://arxiv.org/abs/1711.03209}{arXiv:1711.03209}.

\bibitem{Seidel2} P. Seidel, `Graded Lagrangian submanifolds', \textit{Bull. Soc. Math. France} \textbf{128} (2000), 103-149.

\bibitem{Seidel1} P. Seidel, `Lectures on categorical dynamics and symplectic topology', preprint, 2013, \href{http://math.mit.edu/seidel/937/lecture-notes.pdf}{http://math.mit.edu/seidel/937/lecture-notes.pdf}.

\bibitem{SYZ} A. Strominger, S.-T. Yau and E. Zaslow, `Mirror symmetry is T-duality', \textit{Nuclear Phys. B} \textbf{479}(1-2) (1996), 243-259. 

\bibitem{Tu}
J. Tu, `On the reconstruction problem in mirror symmetry', \textit{Adv. Math.} \textbf{256} (2014), 449-478.

\bibitem{Yuan}
H. Yuan, `Family Floer program and non-archimedean SYZ mirror construction', 2020, \href{https://arxiv.org/abs/2003.06106v1}{arXiv:2003.06106}.

\end{thebibliography}
\end{document}